\documentclass[%
]{mpi2015-cscpreprint}

\usepackage[american]{babel}

\usepackage{graphicx}

\usepackage{amssymb}
\usepackage{amsthm}
\usepackage{mymacros}
\usepackage{array}

\usepackage{mathptmx}
\usepackage{amsmath}
\usepackage{amsfonts}
\usepackage{amssymb}
\usepackage{amsthm}

\usepackage{booktabs}

\usepackage{placeins}

\usepackage{algorithm}
\usepackage{algorithmicx}
\usepackage{algpseudocode}

\usepackage{subcaption}
 \usepackage{numprint}

\usepackage{pgfplots}
\pgfplotsset{compat=1.13}
\usetikzlibrary{shapes.multipart,decorations.markings,decorations.text, 
decorations.pathreplacing,patterns,positioning,intersections,calc,matrix}

\newtheorem{theorem}{Theorem}
\newtheorem{lemma}[theorem]{Lemma}

\begin{document}
  

\title{Efficient and Accurate Algorithms for Solving the Bethe-Salpeter Eigenvalue Problem for Crystalline Systems}
  
\author[$\dagger$]{Peter Benner}
\affil[$\dagger$]{Max Planck Institute for Dynamics of Complex Technical Systems, \authorcr
    Sandtorstr. 1, 39106 Magdeburg, Germany.}
  
\author[$\dagger\ast$]{Carolin Penke}
\affil[$\ast$]{Corresponding author.  \email{penke@mpi-magdeburg.mpg.de}}
  
\shorttitle{New Algorithms for Bethe-Salpeter Eigenvalue problems}
\shortauthor{P. Benner, C.Penke}
  
\keywords{Bethe-Salpeter, Many-body Perturbation Theory, Structured Eigenvalue Problem, Efficient Algorithms, Matrix Square Root, Cholesky Factorization, Singular Value Decomposition}

\msc{15A18, 65F15}
  
\abstract{%
Optical properties of materials related to light absorption and scattering are explained by the excitation of electrons. The Bethe-Salpeter equation is the state-of-the-art approach to describe these processes from first principles (\emph{ab initio}), i.e. without the need for empirical data in the model. To harness the predictive power of the equation, it is mapped to an eigenvalue problem via an appropriate discretization scheme. The eigenpairs of the resulting large, dense, structured matrix can be used to compute dielectric properties of the considered crystalline or molecular system. The matrix always shows a $2\times 2$ block structure. Depending on exact circumstances and discretization schemes, one ends up with a matrix structure such as 
\begin{alignat*}{2}
 H_1  &= \begin{bmatrix}
         A & B \\
         -B & -A
        \end{bmatrix}\in\mathbb{C}^{2n\times 2n},\qquad &A=A^{\herm},\ B=B^{\herm},
        \\
 \text{or }\qquad H_2  &= \begin{bmatrix}
         A & B \\
         -B^{\herm} & -A^{\tran}
        \end{bmatrix}\in\mathbb{C}^{2n\times 2n} \text{ or } \mathbb{R}^{2n\times 2n},\qquad &A=A^{\herm},\ B=B^{\tran}.
\end{alignat*}
Additionally, certain definiteness properties typically hold. $H_1$ can be acquired for crystalline systems \cite{SanMK14},  $H_2$ is a more general form found e.g. in \cite{ShaDYetal16} and \cite{PenMVetal20}, which can for example be used to study molecules.
In this work, we present new theoretical results characterizing the structure of $H_1$ and $H_2$ in the language of non-standard scalar products. These results enable us to develop a new perspective on the state-of-the-art solution approach for matrices of form $H_1$. This new viewpoint is used to develop two new methods for solving the eigenvalue problem. One requires less computational effort while providing the same degree of accuracy. The other one improves the expected accuracy, compared to methods currently in use, with a comparable performance. Both methods are well suited for high performance environments and only rely on basic numerical linear algebra building blocks.}
  
\maketitle

  
\section{Introduction and Preliminaries}
The accurate and efficient computation of optical properties of molecules and condensed matter has been an objective actively pursued in recent years \cite{SanMK14, VorACetal19, SagA09}. In particular, the increasing importance of renewable energies reinforces the interest in the \textit{in silico} prediction of optical properties of novel composite materials and nanostructures. 

New theoretical and algorithmic developments need to go hand in hand with the ever advancing computer technology. In view of the ongoing massive increase in parallel computing power \cite{top500}, the solution of problems that were considered almost impossible just a few years ago comes within reach. In order to unlock the full potential of a supercomputer, great attention must be paid to the development of parallelizable and reliable methods. 

\textit{Ab initio} spectroscopy aims to compute optical properties of materials from first principles, without the need for empirical parameters.  A state-of-the-art approach is derived from many-body perturbation theory and relies on solving the Bethe-Salpeter equation for the \emph{density fluctuation response function} $P(\omega)$. This function describes the propagation of an electron-hole pair and is used to compute optical properties such as the optical absorption spectrum. Its poles give the excitation energies of the given system  (for details see \cite{SanMK14}). Restricting the number of considered  occupied and unoccupied orbitals, the propagator can be represented in (frequency-dependent) matrix form with respect to a set of resonant and antiresonant two-orbital states. The Bethe-Salpeter equation can be rewritten to show that the inverse of the matrix form of $P(\omega)$ is approximated by the matrix pencil 
\begin{align}\label{Eq:BSEPencil}
 \mathcal{H} - \omega \Sigma,
\end{align}
where $\Sigma=\begin{bmatrix}
               I_n & 0 \\
               0 & -I_n
              \end{bmatrix}$ contains a positive and a negative identity matrix on its diagonal.
 $\mathcal{H}\in\mathbb{C}^{2n\times 2n}$ is a Hermitian matrix. It is computed from the Coulomb interaction and the screened interaction in matrix form (familiar from Hedin's equations \cite{HedL70}) and scalar energy differences between occupied and unoccupied  orbitals. The periodicity of crystalline systems implies a time-inversion symmetry in the basis functions used for the discretization.
This leads to $\mathcal{H}$ having the following form: 
\begin{align}\label{Eq:EVform1}
\mathcal{H}=\mathcal{H}_1 = \begin{bmatrix}
         A & B \\
         B & A
        \end{bmatrix},\qquad A=A^{\herm},\ B=B^{\herm}.
\end{align}

In this paper we do not consider the generalized eigenvalue problem given in \ref{Eq:BSEPencil}. Instead, we focus on the the corresponding standard eigenvalue problem of the matrix
\begin{align}\label{Eq:Hform1}
 H_1:=\Sigma^{-1}\mathcal{H}_1 = \Sigma \mathcal{H}_1 = \begin{bmatrix}
         A & B \\
         -B & -A
        \end{bmatrix},\qquad A=A^{\herm},\ B=B^{\herm},
\end{align}
to which we refer as a \emph{BSE matrix of form I}.

If the time-inversion symmetry in the basis functions is not exploited or not available, i.e. when a non-crystalline system is considered, the resulting structure is slightly different. The matrix considered for a standard eigenvalue problem then  turns out to have the form
\begin{align}\label{Eq:Hform2}
 H_2 := \Sigma \mathcal{H}_2 = \begin{bmatrix}
         A & B \\
         -B^{\herm} & -A^{\tran}
        \end{bmatrix}\in\mathbb{C}^{2n\times 2n} \text{ or } \mathbb{R}^{2n\times 2n},\qquad A=A^{\herm},\ B=B^{\tran},
\end{align}
where $.^{\herm}$ denotes the Hermitian transpose, and $.^{\tran}$ denotes the regular transpose without complex conjugation.
We call matrices of this form \emph{BSE matrix of form II}.

In this paper, we contribute new methods that exploit the structure given in \eqref{Eq:Hform1} and improve upon previous approaches in terms of performance and accuracy. They are well-suited for high performance computing as they rely on basic linear algebra building blocks for which high performance implementations are readily available. 

We characterize the structure of the BSE matrices \eqref{Eq:Hform1} and \eqref{Eq:Hform2} by employing the concept of non-standard scalar products. We introduce the notation and concepts following \cite{MacMT05}.

A nonsingular matrix $M$ defines a scalar product on $\mathbb{C}^{n}$ $\langle.,.\rangle_M$, which is a bilinear or sesquilinear form,  given by

\begin{align*}
\langle x, y \rangle_M = \begin{cases}
                          x^{\tran} My\text{ for bilinear forms,}\\
                          x^{\herm} My\text{ for sesquilinear forms,}
                         \end{cases} 
\end{align*}
for $x,y\in\mathbb{C}^{n}$. For a matrix $A\in\mathbb{C}^{n\times n}$,  $A^{\star_M}\in\mathbb{C}^{n\times n}$ denotes the adjoint with respect to the scalar product defined by $M$. This is a uniquely defined matrix satisfying the identity
\begin{align*}
 \langle Ax, y \rangle_M = \langle x, A^{\star_M}y \rangle_M
\end{align*}
for all $x,y\in\mathbb{C}^{n}$. We call $A^{\star_M}$ the $M$-adjoint of $A$ and it holds
\begin{align}\label{Eq:StarM}
 A^{\star_M} = M^{-1}A^*M,
\end{align}
where $.^*$ can refer to the transpose $.^{\tran}$ or Hermitian transpose $.^{\herm}$, depending on whether a bilinear or a sesquilinear form is considered. 
We call the matrix $M$-orthogonal if $A^{\star_M} = A^{-1}$ (given the inverse exists), $M$-self-adjoint if $A=A^{\star_M}$ and $M$-skew-adjoint if $A=-A^{\star_M}$.
\section{Results on the spectral structure of BSE matrices}\label{Sec:Theory}
We now have a language to describe the structure of the BSE matrices in a more concise way, not relying on the matrix block structure. The following two matrices, and the scalar products induced by them, play a central role:
\begin{align}
 J_n = \begin{bmatrix}
 0 & I_n\\
 -I_n & 0
\end{bmatrix},\qquad
\Sigma_n = \begin{bmatrix}
            I_n &0 \\
            0& -I_n
           \end{bmatrix}.
\end{align}
We drop the index when the dimension is clear from its context. The identities $J^{-1} = -J$ and $\Sigma^{-1}=\Sigma$ are regularly used in the following.

\begin{theorem}\label{Thm:Form1}
 A matrix $H\in\mathbb{C}^{2n\times 2n}$ is a BSE matrix of form I as given in \eqref{Eq:Hform1} if and only if both of the following conditions hold.
 \begin{enumerate}
  \item $H$ is skew-adjoint with respect to the complex sesquilinear form induced by $J$, i.e. $H = -H^{\star_J} = JH^{\herm}J$.
  \item $H$ is self-adjoint with respect to the complex sesquilinear form induced by $\Sigma$, i.e. $H=H^{\star_\Sigma} = \Sigma H^{\herm} \Sigma$.
 \end{enumerate}
\end{theorem}
\begin{proof}
$JH^{\herm}J = H$ is equivalent to  $JH$ being Hermitian and $\Sigma H^{\herm} \Sigma =H$ is equivalent to $\Sigma H$ being Hermitian. We observe that $JH$ and $\Sigma H$ are Hermitian, if $H$ has BSE form I. Conversely, let  $H=\begin{bmatrix}
        H_{11} & H_{12}\\
        H_{21} & H_{22}
       \end{bmatrix},\quad H_{ij}\in\mathbb{C}^{n\times n}$ and $JH = \begin{bmatrix}
       H_{21} & H_{22}\\
       -H_{11} & -H_{12}
      \end{bmatrix}$ be Hermitian. It follows
\begin{align}
 H_{21}=H_{21}^{\herm}, \quad H_{12}=H_{12}^{\herm}, \quad H_{11}=-H_{22}^{\herm}.\label{BSEJK1}
\end{align}
Let $\Sigma H = \begin{bmatrix}
        H_{11} & H_{12}\\
        -H_{21} & -H_{22}
       \end{bmatrix}$ be Hermitian. It follows
\begin{align}
 H_{11}=H_{11}^{\herm},\quad H_{22}=H_{22}^{\herm},\quad H_{21}=-H_{12}^{\herm}.\label{BSEJK2}
\end{align}
\eqref{BSEJK1} and \eqref{BSEJK2} give exactly BSE form I:
\begin{align*}
 H= \begin{bmatrix}
     H_{11} & H_{12}\\
     -H_{12}^{\herm} & -H_{11}^{\herm}
    \end{bmatrix}\text{ with } H_{11}=H_{11}^{\herm},\ H_{12}=H_{12}^{\herm}.
\end{align*}
\end{proof}

\begin{theorem}\label{Thm:Form2}
 A matrix $H$ is a BSE matrix of form II as given in \eqref{Eq:Hform2} if and only if both of the following conditions hold.
 \begin{enumerate}
  \item $H$ is skew-adjoint with respect to the complex bilinear form induced by $J$, i.e. $H = -H^{\star_J} = JH^{\tran}J$.
  \item $H$ is self-adjoint with respect to the complex sesquilinear form induced by $\Sigma$, i.e. $H=H^{\star_\Sigma} = \Sigma H^{\herm} \Sigma$.
 \end{enumerate}
\end{theorem}
\begin{proof}
The proof works exactly as the proof of Theorem \ref{Thm:Form1}, but here, the complex transpose $.^{\tran}$ is associated with $J$ instead of the Hermitian transpose $.^{\herm}$.
\end{proof}

Using this new characterization, we see that eigenvalues and eigenvectors also exhibit special structures. Matrices, that are skew-adjoint with respect to the sesquilinear form induced by $J$ are called \emph{Hamiltonian}, and play an important role in control theory and model order reduction (see e.g. \cite{BenKM05}).  The same property with respect to the bilinear form is called \emph{J-symmetric} in \cite{MacMT05} and explored further in \cite{BenFY18}.

The first two propositions of the following theorem are well known facts about Hamiltonian \cite{BenKM05} and J-symmetric matrices \cite{BenFY18}.  Here, $\bar{\lambda}$ denotes the complex conjugate of $\lambda$.

\begin{theorem}\label{Thm:EValSym}
 Let $H\in\mathbb{C}^{2n\times 2n}$.
 
 \begin{enumerate}
  \item If $H$ is skew-adjoint with respect to the sesquilinear form induced by $J$, i.e. $JH = -H^{\herm}J$, then its eigenvalues come in pairs $(\lambda, -\bar{\lambda})$. If $x$ is a right eigenvector of $H$ corresponding to $\lambda$, then $x^{\herm}J$ is the left eigenvector of $H$ corresponding to $-\bar{\lambda}$.
  \item If $H$ is skew-adjoint with respect to the bilinear form induced by $J$, i.e. $JH = -H^{\tran}J$, then its eigenvalues come in pairs $(\lambda, -{\lambda})$. If $x$ is a right eigenvector of $H$ corresponding to $\lambda$, then $x^{\tran}J$ is the left eigenvector of $H$ corresponding to $-{\lambda}$.
  \item If $H$ is self-adjoint with respect to the sesquilinear form induced by $\Sigma$, i.e. $\Sigma H = H^{\herm}\Sigma$, then eigenvalues come in pairs $(\lambda, \bar{\lambda})$. If $x$ is a right eigenvector of $H$ corresponding to $\lambda$, then $x^{\herm}\Sigma$ is the left eigenvector of $H$ corresponding to $\bar{\lambda}$.
 \end{enumerate}
\end{theorem}
\begin{proof}
\begin{enumerate}
 \item Using $H^{\herm}=JHJ$ and $J^{-1}=-J$, we see that
  \begin{align*}
  Hx =\lambda x\quad
  \Leftrightarrow \quad x^{\herm} J H  = - \bar{\lambda} x^{\herm} J.
 \end{align*}

 \item Using $H^{\tran}=JHJ$ and $J^{-1}=-J$, we see that
 \begin{align*}
  Hx =\lambda x\quad
  \Leftrightarrow \quad x^{\tran} J H  = - \lambda x^{\tran} J.
 \end{align*}
 \item Using $H^{\herm}=\Sigma H \Sigma $ and $\Sigma^{-1}=\Sigma$, we see that
 \begin{align*}
  Hx =\lambda x\quad
  \Leftrightarrow \quad x^{\herm} \Sigma H  = \bar{\lambda} x^{\herm} \Sigma.
 \end{align*}

\end{enumerate}

\end{proof}

Theorem \ref{Thm:EValSym} reveals that symmetries defined by the matrices $J$ or $\Sigma$ are reflected in connections between left and right eigenvectors of the considered matrix. The BSE matrices show a symmetry with respect to two scalar products (Theorem \ref{Thm:Form1} and \ref{Thm:Form2}). This double-structure leads to eigenvalues that show up not only in pairs but in quadruples if they have a real and an imaginary component. Additionally, it yields a connection between right eigenvectors, clarified in the following theorem.

\begin{theorem}
 Let $H\in\mathbb{C}^{2n\times 2n}$ be self-adjoint with respect to $\Sigma$ and skew-adjoint with respect to (a) the sesquilinear scalar product or (b) the bilinear scalar product induced by $J$. Then
 \begin{enumerate}
  \item The eigenvalues of $H$ come in pairs $(\lambda,-\lambda)$ if $\lambda\in\mathbb{R}$ or $\lambda\in\iu\mathbb{R}$, or in quadruples $(\lambda,\bar{\lambda},-\lambda,-\bar{\lambda})$.
  \item 
  \begin{enumerate}
  \item If $v$ is an eigenvector of $H$ with respect to $\lambda$, then $J\Sigma v$ is an eigenvector of $H$ with respect to $-\lambda$. 
  \item If $v$ is an eigenvector of $H$ with respect to $\lambda$, then $J\Sigma \bar{v}$ is an eigenvector of $H$ with respect to $-\bar{\lambda}$. 
  \end{enumerate}
 \end{enumerate}
\end{theorem}

\begin{proof}
 \begin{enumerate}
  \item The quadruple property comes from combining the propositions given in Theorem \ref{Thm:EValSym} (1. and 3. or 2. and 3., respectively). The pair property for real eigenvalues comes (a) from Theorem \ref{Thm:EValSym}, proposition 1, or (b) from Theorem \ref{Thm:EValSym}, proposition 2. The pair property for imaginary eigenvalues follows from Theorem \ref{Thm:EValSym}, proposition 3 in both cases.
  \item 
  \begin{enumerate}
  \item  With $\Sigma J H J \Sigma = H$ we have
  \begin{align*}
   Hv = \lambda v \quad\Leftrightarrow\quad HJ\Sigma v = -\lambda J\Sigma v.
  \end{align*} 
  \item  With $\Sigma J H J \Sigma = \bar{H}$ we have
  \begin{align*}
   Hv = \lambda v \quad \Leftrightarrow\quad  \bar{H}J\Sigma v = -\lambda J\Sigma v\quad \Leftrightarrow \quad HJ\Sigma \bar{v} = -\bar{\lambda}J\Sigma\bar{v}.
  \end{align*}
  \end{enumerate}
 \end{enumerate}

\end{proof}

 The special case of (b) (i.e. for BSE matrices of form II) has been proven in \cite{BenFY18}. Our proof does not rely on the particular block structure of the matrix, but works with the given symmetries and is therefore more concise and easily extendable to other double-structured matrices. 

 In the practice of computing excitation properties of materials, there is even more structure available. It can be exploited for devising efficient algorithms. The Hermitian matrices 
 \begin{align}\label{Eq:BSEHam}
 \mathcal{H}_1 := \Sigma H_1 = \begin{bmatrix}
                A & B \\
                B & A
               \end{bmatrix},\qquad
  \mathcal{H}_2 :=\Sigma H_2 = \begin{bmatrix}
                A & B \\
                B^{\herm} & A^{\tran}
               \end{bmatrix},
 \end{align}
 introduced in \eqref{Eq:BSEPencil}, are called BSE Hamiltonians. They are typically positive definite \cite{OniRR02,SanMK14}.  The term ``Hamiltonian'' might cause confusion in this context, as it has a different meaning in numerical linear algebra and in electronic structure theory. We have used the numerical linear algebra meaning, by which a matrix $H$ is called Hamiltonian if it holds $JH=-JH^{\herm}$. In electronic structure theory, the term ``Hamiltonian'' is inspired by the Hamiltonian operator from basic quantum mechanics. It refers to the left matrix $\mathcal{H}$ of a generalized, Schr\"odinger-like, eigenvalue problem, such as \eqref{Eq:BSEPencil}, which is typically Hermitian. 


The definiteness property has consequences for the structure of the eigenvalue spectrum. To study these, we consider the more general class of $\Sigma$-Hermitian matrices, by which we mean matrices that are self-adjoint with respect to the scalar product induced by a signature matrix $\Sigma$. 

\begin{theorem}\label{Thm:DefiniteRealEvals}
Let $\Sigma = \diag{\sigma_1,\dots,\sigma_n}$, $\sigma_i\in\{+1,-1\}$ be a signature matrix with $p$ positive and $n-p$ negative diagonal entries. Let $H\in\mathbb{C}^{n\times n}$ be given such that $\Sigma H$ is Hermitian positive definite. Then $H$ is diagonalizable and its eigenvalues are real, of which $p$ are positive and $n-p$ are negative.
\end{theorem}

\begin{proof}
  As $\Sigma H$ is positive definite, and $\Sigma$ is symmetric, they can be diagonalized simultaneously (see \cite{GolV13}, Corollary 8.7.2), i.e. there is a nonsingular $X\in\mathbb{C}^{n\times n}$, s.t
 \begin{align}
  X^{\herm}\Sigma H X &= I_n, \label{Eq:SimDiag1}\\
  X^{\herm} \Sigma X &= \Lambda \in\mathbb{R}^{n\times n}, \label{Eq:SimDiag2}
 \end{align}
 where $\Lambda=\diag{\lambda_1,\dots,\lambda_n}$ gives the eigenvalues of the matrix pencil $\Sigma x - \lambda \Sigma H$.  It follows from \eqref{Eq:SimDiag2} and Sylvester's law of inertia that $\Lambda$ has $p$ positve and $n-p$ negative values. We have
 \begin{align}
  X^{-1}HX= \Lambda^{-1},
 \end{align}
 i.e. $H$ is diagonalizable and $\Lambda^{-1}$ contains the eigenvalues of $H$. 
\end{proof}

The spectral structure of the BSE matrices given in practice follows immediately from the presented theorems and is summarized in the following lemma.

\begin{lemma}\label{Lem:ComputeNegEvecs}
 Let $H$ be a BSE matrix of form I (see \eqref{Eq:Hform1}) or form II (see \eqref{Eq:Hform2}), such that the BSE Hamiltonian \eqref{Eq:BSEHam} is positive definite. Then the eigenvalues are real and come in pairs $\pm\lambda$. If $v$ is an eigenvector associated with $\lambda$, then
 \begin{enumerate}
  \item $y=\begin{bmatrix} 0 & I \\                                                                                                                                                                                                                                                                                 I & 0                                                                                                                                                                                                                                                                                \end{bmatrix}v$ is an eigenvector associated with $-\lambda$ if $H$ is of BSE form I,
  \item $y=\begin{bmatrix} 0 & I \\                                                                                                                                                                                                                                                                                 I & 0                                                                                                                                                                                                                                                                                \end{bmatrix}\bar{v}$ is an eigenvector associated with $-\lambda$ if $H$ is of BSE form II.
 \end{enumerate}

\end{lemma}

In the remaining part of this section we focus on BSE matrices of form I, paving the way for new, efficient algorithms. Two essential observations to make the problem more tractable are the following, which can e.g. be found in \cite{ShaDYetal16}, albeit for real matrices.

\begin{lemma}\label{Lem:Qtransform}
 Let  $H$ be a BSE matrix of form I \eqref{Eq:Hform1}. 
 \begin{enumerate}
  \item With the matrix $Q=\frac{1}{2}\begin{bmatrix}
          I & I\\
          -I & I
         \end{bmatrix}$ we have 
\begin{align}\label{Eq:TransMatrix1}
Q^{-1} H Q = \begin{bmatrix}
           0 &A-B\\
           A +B & 0
          \end{bmatrix}.
\end{align}
\item $\Sigma H$ is positive definite  if and only if $A+B$ and $A-B$ are positive definite.
 \end{enumerate}
\end{lemma}

The following theorem plays the central role in this paper. We use it as a new tool to derive an existing solution approach. Within this new framework, other solution approaches will become apparent, yielding significant benefits compared to the existing approach.

\begin{theorem}\label{Thm:ProdEig}
 Let $H$ be a BSE matrix of form I (see \eqref{Eq:Hform1}) wih a positive definite BSE Hamiltonian \eqref{Eq:BSEHam}. Then  $M_1:=A+B$, $M_2=A-B$ are Hermitian positive definite. 
Let 
\begin{align}
 M_1 M_2 v_1 =  \mu v_1, \qquad
 v_2^{\herm} M_1 M_2 = \mu v_2^{\herm}, \qquad v_1^{\herm}v_2 = 1, \label{Eq:leftrightEV}
\end{align}
define a pair of right and left eigenvectors of the matrix product $M_1 M_2$. Then $\mu\in\mathbb{R}$, $\mu>0$. With $Q:=\frac{1}{2}\begin{bmatrix}
          I & I\\
          -I & I
         \end{bmatrix}$ and scaling factors $\lambda_1:={v_2^{\herm} M_1 v_2}>0$, $\lambda_2:={v_1^{\herm} M_2 v_1}>0$, an eigenpair of $H$ is given by
\begin{align}\label{Eq:GenScaling}
 \lambda= \sqrt{\mu},\qquad v_{\lambda} = Q \begin{bmatrix}
        v_1 \lambda_1^{\frac{1}{4}}\lambda_2^{-\frac{1}{4}}\\
        v_2 \lambda_1^{-\frac{1}{4}}\lambda_2^{\frac{1}{4}}
       \end{bmatrix},
\end{align}
 i.e. $Hv_{\lambda} = \lambda v_{\lambda}$. If $\gamma$ is another eigenvalue of $M_1M_2$ and $v_\gamma$ is the corresponding constructed vector, it holds
 \begin{align}\label{Eq:SigmaOrthEVecs}
  v_\lambda ^{\herm}\Sigma v_{\gamma}  = \begin{cases}
                                    1 & \text{if } \lambda=\gamma\\
                                    0 & \text{if } \lambda\neq\gamma.\\
                                   \end{cases}
 \end{align}

\end{theorem}

\begin{proof}
We observe $(Q^{-1}HQ)^2 = \diag{M_1M_2,M_2M_1}$. Therefore the eigenvalues of $M_1M_2$ must be a subset of the eigenvalues of $H^2$. These are positive real according to Theorem \ref{Thm:DefiniteRealEvals}. Note that $v_2$ being a left eigenvector of $M_1M_2$ is equivalent to $v_2$ being a right eigenvector of $M_2M_1$, because $M_1$ and $M_2$ are Hermitian. It follows from \eqref{Eq:leftrightEV} that
\begin{align*}
 M_1M_2M_1 v_2 =\mu M_1 v_2,
\end{align*}
i.e. $M_1v_2$ lies in the eigenspace of $M_1M_2$ corresponding to $\mu$. 
If $M_1v_2$ is not a multiple of $v_1$ it must hold $(M_1v_2)^{\herm{}}v_2 = 0$, which contradicts the fact, that $M_1$ is positive definite. So there is $\lambda_1\in\mathbb{C}$, s.t.
\begin{align}\label{Eq:crossEV1}
 M_1v_2 = \lambda_1 v_1.
\end{align}
Similarly it follows from \eqref{Eq:leftrightEV} that there is $\lambda_2\in\mathbb{C}$, s.t.
\begin{align}\label{Eq:crossEV2}
 M_2v_1 = \lambda_2 v_2.
\end{align}
Following from \eqref{Eq:crossEV1} and \eqref{Eq:crossEV2} and using $v_1^{\herm}v_2=1$, $\lambda_1$ and $\lambda_2$ can be computed as
\begin{align}
 \lambda_1 = v_2^{\herm} M_1 v_2,\qquad \lambda_2=v_1^{\herm} M_2 v_1.
\end{align}
It follows that $\lambda_1,\lambda_2\in\mathbb{R}_+$ because $M_1$ and $M_2$ are positive definite. Inserting \eqref{Eq:crossEV2} in \eqref{Eq:crossEV1} we get
\begin{align}
 M_1M_2v_1 = \lambda_1\lambda_2 v_1.
\end{align}
With \eqref{Eq:leftrightEV} it follows $\lambda_1\lambda_2=\mu$. For $\lambda \in\mathbb{C}$, 
\begin{align}
 \begin{bmatrix} 0 & M_1 \\
         M_2 &0
        \end{bmatrix} \begin{bmatrix}x\\y\end{bmatrix}
        = \lambda \begin{bmatrix}x\\y\end{bmatrix}
\end{align}
holds if and only if
\begin{align}
 M_1 y = \lambda x\quad \text{and} \quad M_2 x = \lambda y.
\end{align}
This is achieved by $\lambda:=\sqrt{\lambda_1\lambda_2}=\sqrt{\mu}$ and $x:=v_1 \lambda_1^{\frac{1}{4}}\lambda_2^{-\frac{1}{4}}$ and $y:=v_2\lambda_1^{-\frac{1}{4}}\lambda_2^{\frac{1}{4}}$, following from \eqref{Eq:crossEV1} and \eqref{Eq:crossEV2}.

In conclusion we have
 \begin{align*}
  Hv_{\lambda} = HQ\begin{bmatrix}
         x\\y
        \end{bmatrix}=
 Q\begin{bmatrix} 0 & M_1 \\
         M_2 &0
        \end{bmatrix}\begin{bmatrix}
        x\\y
       \end{bmatrix} = \lambda Q\begin{bmatrix}x\\y\end{bmatrix} = \lambda v_\lambda.
 \end{align*}
The $\Sigma$-orthogonality condition \eqref{Eq:SigmaOrthEVecs} remains to be shown. We observe $Q^{\herm}\Sigma Q = \frac{1}{2}\begin{bmatrix}
                                                                                                                          0 & I\\
                                                                                                                          I & 0\end{bmatrix}$ and with 
\begin{align}
v_\lambda =Q \begin{bmatrix}x_\lambda\\y_\lambda\end{bmatrix},\quad 
x_\lambda=v_{1,\lambda}\lambda_{1}^{\frac{1}{4}}\lambda_{2}^{-\frac{1}{4}},\quad 
y_\lambda=v_{2,\lambda}\lambda_{1}^{-\frac{1}{4}}\lambda_{2}^{\frac{1}{4}},\\
v_\gamma = Q\begin{bmatrix}x_\gamma\\y_\gamma\end{bmatrix},\quad 
x_\gamma=v_{1,\gamma}\gamma_{1}^{\frac{1}{4}}\gamma_{2}^{-\frac{1}{4}},\quad 
y_\gamma=v_{2,\gamma}\gamma_{1}^{-\frac{1}{4}}\gamma_{2}^{\frac{1}{4}},
\end{align}
we see
\begin{align}\label{Eq:2VecSigmaOrth}
 v_\lambda^{\herm} \Sigma v_\gamma = \frac{1}{2} 
 (v_{2,\lambda}^{\herm}v_{1,\gamma} \bar{\lambda}_1^{-\frac{1}{4}}\bar{\lambda}_2^{\frac{1}{4}}\gamma_1^{\frac{1}{4}}\gamma_2^{-\frac{1}{4}}
 +
 v_{1,\lambda}^{\herm}v_{2,\gamma} \bar{\lambda}_1^{\frac{1}{4}}\bar{\lambda}_2^{-\frac{1}{4}}\gamma_1^{-\frac{1}{4}}\gamma_2^{\frac{1}{4}}
 ).
\end{align}
This expression is equal to 0, if $v_{2,\lambda}$ and $v_{1,\gamma}$ are left and right eigenvectors corresponding to different eigenvalues $\lambda\neq \gamma$ of $M_1M_2$.  If $\lambda=\gamma$ and if $v_{\lambda}$ and $v_{\gamma}$ were constructed in the same way, we have $\lambda_1=\gamma_1$ and $\lambda_2=\gamma_2$. Using that $\lambda_1$ and $\lambda_2$ are real, \eqref{Eq:2VecSigmaOrth} simplifies to
\begin{align}
 v_\lambda^{\herm}\Sigma v_\lambda = \frac{1}{2}(v_{2,\lambda}^{\herm}v_{1,\lambda} + v_{1,\lambda}^{\herm}v_{2,\lambda}) = 1,
\end{align}
where we used the normalization of the vectors given in \eqref{Eq:leftrightEV}.
\end{proof}


\section{Algorithms for crystalline systems (Form I)}\label{Sec:Algs}
We have seen in Theorem \ref{Thm:ProdEig} that the Bethe-Salpeter eigenvalue problem of form I with size $2n\times 2n$ can be interpreted as a product eigenvalue problem with two Hermitan factors of size $n\times n$. In practice, the complete set of eigenvectors provides additional insight to excitonic effects. To compute them, left and right eigenvectors of the smaller product eigenvalue problem are needed. Product eigenvalue problems are well studied (see e.g. \cite{Kre05}). A general way to solve these problems, taking the product structure into account to improve numerical properties, is the periodic QR algorithm. This tool can be used for solving general Hamiltonian eigenvalue problems \cite{BenMX98a}. In this work we focus on non-iterative methods that work for Hermitian factors and transform the problem such that it can be treated using available HPC libraries.

The algorithms presented in this section compute the positive part of the spectrum of a BSE matrix and the corresponding eigenvectors. If the eigenvectors corresponding to the negative mirror eigenvalues are of interest, they can easily be computed by employing Lemma \ref{Lem:ComputeNegEvecs}.

\subsection{Square root approach}\label{Sec:Sqrt}
A widely used approach for solving the Bethe-Salpeter eigenvalue problem of form I (e.g. in \cite{SanMK14}) relies on the computation of the matrix square root of $M_2=A-B$. We present it in the following, relating it to the framework given by Theorem \ref{Thm:ProdEig}. 

Starting with the first equation of \eqref{Eq:leftrightEV}, we see
\begin{align}
 M_1 M_2 V_1 &= V_1 \Lambda^2\\
 \Leftrightarrow M_2^{\frac{1}{2}}M_1{M_2}^{\frac{1}{2}} \hat{V} &= \hat{V} \Lambda^2,\label{Eq:SqrtEvecs}
\end{align}
where $\hat{V} = M_2^{\frac{1}{2}}V_1$ contains the eigenvectors of the Hermitian matrix ${M_2}^{\frac{1}{2}}M_1{M_2}^{\frac{1}{2}}$.

On the other hand, with the second equation of \eqref{Eq:leftrightEV}, we see
\begin{align}
 V_2^{\herm} M_1 M_2 &= \Lambda^2 V_2^{\herm} \\
\Leftrightarrow \hat{V}^{\herm{}}{M_2}^{\frac{1}{2}} M_1 {M_2}^{\frac{1}{2}} &= \Lambda^2 \hat{V}^{\herm}
\end{align}
where $\hat{V}={M_2}^{-\frac{1}{2}} V_2 $ contains the left eigenvectors of ${M_2}^{\frac{1}{2}}M_1{M_2}^{\frac{1}{2}}$. Here, note that because ${M_2}^{\frac{1}{2}}M_1{M_2}^{\frac{1}{2}}$ is Hermitian, left and right eigenvectors coincide, we can denote both left and right eigenvector matrices by $\hat{V}$. 

If the generalized eigenvalue problem \eqref{Eq:leftrightEV} is solved in this particular way, we can say more about the resulting scalar factors $\lambda_1$ and $\lambda_2$ in \eqref{Eq:crossEV1} and  \eqref{Eq:crossEV2}.

\begin{lemma}\label{Lem:CrosScalarFactors}
Let $M_1$ and $M_2$ be given as in Theorem \ref{Thm:ProdEig}, $\Lambda^2$ be a diagonal matrix containing the eigenvalues of $M_1M_2$ and $\hat{V}$ contain the eigenvectors of ${M_2}^{\frac{1}{2}}M_1{M_2}^{\frac{1}{2}}$ and $\hat{V}^{\herm}\hat{V}=I$. Then $V_1:={M_2}^{-\frac{1}{2}}\hat{V}$, $V_2:={M_2}^{\frac{1}{2}}\hat{V}$ fulfill
\begin{align}
 V_1^{\herm}V_2 = I,\qquad
 M_2 V_1 =  V_2,  \qquad
 M_1 V_2 = V_1\Lambda^2.
\end{align}
\end{lemma}
\begin{proof}
The first statement is immediately obvious from the normalization of $\hat{V}$ and because $M_2$ is Hermitian.
Starting with \eqref{Eq:SqrtEvecs} we see 
\begin{align*}
 {M_2}^{\frac{1}{2}}M_1{M_2}^{\frac{1}{2}} \hat{V} = \hat{V} \Lambda^2
\Leftrightarrow M_1 V_2 = V_1 \Lambda^2.
\end{align*}
Using $M_1 V_2 = V_1 \Lambda^2$, \eqref{Eq:SqrtEvecs} also yields 
\begin{align*}
 {M_2}^{\frac{1}{2}}M_1{M_2}^{\frac{1}{2}} \hat{V} = \hat{V} \Lambda^2
 \Leftrightarrow 
 M_2V_1 = V_2.
\end{align*}

\end{proof}

Lemma \ref{Lem:CrosScalarFactors} states that the scaling factors in \eqref{Eq:crossEV1} and \eqref{Eq:crossEV2} are given by  $\lambda_1=\lambda^2$ and $\lambda_2=1$ in this case. Theorem \ref{Thm:ProdEig} therefore suggests to scale the acquired eigenvectors in order to obtain eigenvectors of the full matrix as $v := Q \begin{bmatrix}
        v_1 {\lambda}^{\frac{1}{2}}\\
        v_2{\lambda}^{-\frac{1}{2}}
       \end{bmatrix}$.

These observations suggest Algorithm \ref{Alg:H1sqrt}, where the left and right eigenvectors needed in Theorem \ref{Thm:ProdEig} are computed from the eigenvectors of the Hermitian matrix ${M_2}^{\frac{1}{2}}M_1{M_2}^{\frac{1}{2}}$.

\begin{algorithm}[ht]
\caption{Compute eigenvectors of BSE matrix of form I, using the matrix square root.\label{Alg:H1sqrt}}
\begin{algorithmic}[1]
\Require $A=A^{\herm}\in\mathbb{C}^{n\times n}$, $B=B^{\herm}\in\mathbb{C}^{n\times n}$, defining a BSE matrix of form I: $H=\begin{bmatrix}
                                                                                                                 A & B\\
                                                                                                                 -B & -A
                                                                                                                \end{bmatrix}$, s.t. $A+B$ and $A-B$ are positive definite.
\Ensure $V\in\mathbb{C}^{2n\times n}$, $\Lambda=\diag{\lambda_1,\dots,\lambda_n}\in\mathbb{R}^{n\times n}_+$ s.t. $HV=V\Lambda$.
\State $S \leftarrow {(A-B)}^{\frac{1}{2}}$ \label{Alg:H1sqrt:Sqrt}
\State Compute eigendecomposition of Hermitian positive definite matrix $M:=S(A+B)S$, i.e. 
\begin{align*}
 MV_M = V_M D,
\end{align*}
where $V_M$ contains the normalized eigenvectors of $M$ and $D=\diag{d_1,\dots,d_n}\in\mathbb{R}_+^{n\times n}$ the eigenvalues of $M$. \label{Alg:H1sqrt:Eig}
\State $\Lambda \leftarrow {D}^{\frac{1}{2}}$
\State $
 V_1 \leftarrow S^{-1} V_M {\Lambda}^{\frac{1}{2}}$ 
\State $
 V_2 \leftarrow S V_M {\Lambda}^{-\frac{1}{2}}$

\State $V\leftarrow\begin{bmatrix}
           \frac{1}{2}(V_1 + V_2)\\
           \frac{1}{2}(V_2 - V_1)
          \end{bmatrix}$.
\end{algorithmic}
\end{algorithm}

The essential work of this algorithm is the computation of the matrix square root (Step \ref{Alg:H1sqrt:Sqrt}) and the solution of a Hermitian eigenvalue problem (Step \ref{Alg:H1sqrt:Eig}). Computing the (principal) square root of a matrix is a nontrivial task (see e.g. \cite{Hig08}, Chapter 6) with a (perhaps surprisingly) high computational demand. Its efficient computation has been an active area of research. Given a Hermitian positive definite matrix $C$, its principal square root $S$, s.t.  $S^2=C$, can be computed by diagonalizing $M = V_C D_C V_C^{\herm}$, and taking the square roots of the diagonal entries of $D_C$. Then $S:=V_C\sqrt{D_C}V_C^{\herm}$. 


The main computational effort of Algorithm \ref{Alg:H1sqrt} is therefore the subsequent solution of two Hermitian eigenvalue problems.


\subsection{Cholesky factorization approach}\label{Sec:Chol}
We now lay out how the product eigenvalue problem given in \eqref{Eq:leftrightEV} can be solved by using Cholesky factorizations. Let the Cholesky factorization  $M_2 = L L^{\herm}$ be given.

Starting with \eqref{Eq:leftrightEV} we see
\begin{align}
 M_1 M_2 V_1 &= V_1 \Lambda^2\label{EVchol_start}\\
 \Leftrightarrow 
 L^{\herm} M_1 L \hat{V} &= \hat{V} \Lambda^2,
\end{align}
where $\hat{V} = L^{H}V_1$ contains the eigenvectors of the Hermitian matrix $L^{\herm} M_1 L$. 

On the other hand, with the second equation of \eqref{Eq:leftrightEV}, we have
\begin{align}
 V_2^{\herm} M_1 M_2 &= \Lambda^2 V_2^{\herm} \\
\Leftrightarrow \hat{V}^{\herm}L^{\herm} M_1 L &= \Lambda^2 \hat{V}^{\herm},
\end{align}
with $\hat{V}= L^{-1} V_2$ as the eigenvectors of $L^{\herm}M_1L$.


The analogy to Lemma \ref{Lem:CrosScalarFactors} is given in the following, which can be proven in the same way.

\begin{lemma}\label{Lem:CholScalarFactors}
Let $M_1$ and $M_2$ be given as in Theorem \ref{Thm:ProdEig}, $M_2=LL^{\herm}$ be a Cholesky decompositiong, $\Lambda^2$ be a diagonal matrix containing the eigenvalues of $M_1M_2$ and $\hat{V}$ contain eigenvectors of $L^{\herm}M_1L$ and $\hat{V}^{\herm}\hat{V}=I$. Then $V_2:=L\hat{V}$, $V_1:=V_2^{-H}=L^{-H}\hat{V}$ fulfill
\begin{align}
 V_1^{\herm}V_2 = I,\qquad
 M_2 V_1 =  V_2,  \qquad
 M_1 V_2 = V_1\Lambda^2.
\end{align}
\end{lemma}

This suggests the same scaling factors as in Algorithm \ref{Alg:H1sqrt}, leading to Algorithm \ref{Alg:H1chol}. The same key idea is used in the standard approach for solving generalized symmetric-definite eigenvalue problems \cite{GolV13}, implemented in various libraries \cite{LAPACK,MarBJetal14}.

\begin{algorithm}[ht]
\caption{Compute eigenvectors of BSE matrix of form I, using a Cholesky factorization.\label{Alg:H1chol}}
\begin{algorithmic}[1]
\Require $A=A^{\herm}\in\mathbb{C}^{n\times n}$, $B=B^{\herm}\in\mathbb{C}^{n\times n}$, defining a BSE matrix of form I $H=\begin{bmatrix}
                                                                                                                 A & B\\
                                                                                                                 -B & -A
                                                                                                                \end{bmatrix}$, s.t. $A+B$ and $A-B$ are positive definite.
\Ensure $V\in\mathbb{C}^{2n\times n}$, $\Lambda=\diag{\lambda_1,\dots,\lambda_n}\in\mathbb{R}^{n\times n}_+$ s.t. $HV=V\Lambda$.
\State Compute Cholesky factorization $LL^{\herm} = A-B$.
\State Compute eigendecomposition for $M=L^{\herm} (A+B) L$ \label{Alg:H1chol:eig1}
\begin{align*}
 MV_M = V_{M} D
\end{align*}
\State $\Lambda \leftarrow {D}^{\frac{1}{2}}$
\State $V_1\leftarrow L^{-H}V_M{\Lambda}^{\frac{1}{2}}$
\State $V_2\leftarrow LV_M{\Lambda}^{-\frac{1}{2}}$

\State $V\leftarrow\begin{bmatrix}
           \frac{1}{2}(V_1 + V_2)\\
           \frac{1}{2}(V_2 - V_1)
          \end{bmatrix}$ .
\end{algorithmic}
\end{algorithm}

Comparing Algorithm \ref{Alg:H1sqrt} to \ref{Alg:H1chol}, we see that the essential work in both algorithms is solving Hermitian $n\times n$ eigenvalue problems. The Cholesky variant (Algorithm \ref{Alg:H1chol}) solves one explicitly at Step \ref{Alg:H1chol:eig1}. The square root variant (Algorithm \ref{Alg:H1sqrt}) solves one for computing the matrix square root, which is then used to set up the matrix for the second eigenvalue problem.  Then both left and right eigenvectors of the product eigenvalue problem can be inferred from the computed ones. 
%
%
%
%
%

\subsection{Singular Value Decomposition approach}\label{Sec:SVD}
Both the square root approach discussed in Section \ref{Sec:Sqrt} and the Cholesky approach discussed in Section \ref{Sec:Chol} compute the squared eigenvalues of the original problem. In the numerical linear algebra community this procedure is well known to limit the attainable accuracy  \cite{Van84b}. 

The methods essentially work on the (transformed) matrix product $M_1M_2$. It corresponds to the squared matrix $H^2$, as
\begin{align*}
Q^{\herm} H^2 Q = \begin{bmatrix}
              M_1M_2&\\
              &M_2M_1
             \end{bmatrix},\qquad \text{where } Q:=\frac{1}{\sqrt{2}}\begin{bmatrix}
          I & I\\
          -I & I
         \end{bmatrix}.
\end{align*}
See also Lemma \ref{Lem:Qtransform} and the proof of Theorem \ref{Thm:ProdEig}. Now the scaling factor of $Q$ is different to ensure its orthogonality. $H$ belongs to the class of Hamiltonian matrices (see Section \ref{Sec:Theory}).  When the eigenvalues are computed from the squared matrix $H^2$, employing a backward-stable method, the computational error can be approximated using first-order perturbation theory \cite{Wil65,Van84b,BenBB00}. It is given as 
\begin{align}\label{Eq:SquaredErr}
 |\lambda - \hat{\lambda} |  
 \approx {\epsilon} 
 \frac{\|H\|_2}{s(\lambda)} \min{\left\{\frac{\|H\|_2}{\lambda },\frac{1}{\sqrt{\epsilon}}\right\}},
\end{align} 
where $\lambda$ denotes an exact eigenvalue of $H$, $\hat{\lambda}$ the corresponding computed value, $s(\lambda)$ the condition number of the eigenvalue, and $\epsilon$ the machine precision. Unless $\lambda$ is very large, the expression is dominated by $\frac{\sqrt{\epsilon}\|H\|_2}{s(\lambda)}$. Essentially, the number of significant digits of the eigenvalues is halved, compared to direct backward-stable methods. For example, applying the QR algorithm on the original matrix $H$ would yield an approximate error of $\frac{\epsilon\|H\|_2}{s(\lambda)}$. It fails however, to preserve and exploit the structure of the problem and is undesirable from a numerical as well as from a performance point of view.


A remedy is given in the approach discussed in this section, making use of the singular value decomposition (SVD).

Given the Cholesy factorizations $L_1L_1^{\herm}=M_1$, $L_2L_2^{\herm}=M_2$ and the SVD $U\Lambda V^{\herm} = L_1L_2^{\herm}$, we observe that $\Lambda$ contains the eigenvalues of the BSE matrix, i.e. the square roots of the eigenvalues of the matrix product $M_1M_2$. The details of the eigenvector computation are given in the following Lemma.

\begin{lemma}\label{Lem:SVDScalarFactors}
Let $M_1$ and $M_2$ be given as in Theorem \ref{Thm:ProdEig}, $L_1L_1^{\herm}=M_1$, $L_2L_2^{\herm}=M_2$ be Cholesky factorizations, and $L_1^{\herm}L_2=U\Lambda V^{\herm}$ be a singular value decomposition. Then $V_1:=L_1U {\Lambda}^{-\frac{1}{2}}$, $V_2:=L_2 V {\Lambda}^{-\frac{1}{2}}$ fulfill
\begin{align}
 V_1^{\herm}V_2 = I,\qquad
 M_2 V_1 =  V_2\Lambda,  \qquad
 M_1 V_2 = V_1\Lambda.
\end{align}
\end{lemma}
\begin{proof}
It holds
\begin{align*}
V_1^{\herm}V_2 = {\Lambda}^{-\frac{1}{2}}U^{\herm}L_1^{\herm}L_2V{\Lambda}^{-\frac{1}{2}} = {\Lambda}^{-\frac{1}{2}}U^{\herm}U\Lambda V^{\herm}V{\Lambda}^{-\frac{1}{2}} = I
\end{align*}
and
\begin{align*}
 M_2V_1 = M_2L_1U{\Lambda}^{-\frac{1}{2}} = L_2L_2^{\herm}L_1U{\Lambda}^{-\frac{1}{2}} = L_2 V\Lambda U^{\herm} U {\Lambda}^{-\frac{1}{2}} = L_2V{\Lambda}^{\frac{1}{2}} = V_2\Lambda.
 \end{align*}
 $M_1V_2=V_1\Lambda$ is proved in the same way.  
\end{proof}

Lemma \ref{Lem:SVDScalarFactors} states that the scaling described in \eqref{Eq:GenScaling} boils down to a scaling factor of 1 as $\lambda_1=\lambda_2$. This suggests Algorithm \ref{Alg:H1svd}.

\begin{algorithm}[ht]
\caption{Compute eigenvectors of BSE matrix of form I, using the singular value decomposition.\label{Alg:H1svd}}
\begin{algorithmic}[1]
\Require $A=A^{\herm}\in\mathbb{C}^{n\times n}$, $B=B^{\herm}\in\mathbb{C}^{n\times n}$, defining a BSE matrix of form I $H=\begin{bmatrix}
                                                                                                                 A & B\\
                                                                                                                 -B & -A
                                                                                                                \end{bmatrix}$, s.t. $=A+B$ and $A-B$ are positive definite.
\Ensure $V\in\mathbb{C}^{2n\times n}$, $\Lambda=\diag{\lambda_1,\dots,\lambda_n}\in\mathbb{R}^{n\times n}_+$ s.t. $HV=V\Lambda$.
\State Compute Cholesky factorization $L_1L_1^{\herm} = M_1$.
\State Compute Cholesky factorization $L_2L_2^{\herm} = M_2$.
\State Compute singular value decomposition $U_{SVD}\Lambda V_{SVD}^{\herm} = L_1^{\herm}L_2$
\State $V_1 \leftarrow L_1U_{SVD} {\Lambda}^{-\frac{1}{2}}$
\State $V_2 \leftarrow L_2 V_{SVD}{\Lambda}^{-\frac{1}{2}}$
\State $V\leftarrow\begin{bmatrix}
           \frac{1}{2}(V_1+V_2)\\
           \frac{1}{2}(V_2-V_1)
          \end{bmatrix}$.
\end{algorithmic}
\end{algorithm}

The main difference between the SVD-based algorithm (Algorithm \ref{Alg:H1svd}) and the other ones, from a numerical point of view, is that the eigenvalue matrix $\Lambda$ is computed directly by the SVD and not as a square root of another diagonal matrix $D$.

The way real BSE matrices of form II \eqref{Eq:Hform2} are treated in \cite{ShaDYetal16} is based on the same idea. 



We can expect to see a higher accuracy in the eigenvalues than in the square root and the Cholesky approach, because the eigenvalues are computed directly, using a backward-stable method for the singular value decomposition. Perturbation theory \cite{SteS90} yields an approximate error of
\begin{align}
 |\lambda-\hat{\lambda}| \approx {\epsilon} 
 \frac{\|H\|_2}{s(\lambda)}.
\end{align}
In the other approaches, a similar approximation only holds for the error of the squared eigenvalues $\lambda^2$, and translates in form of \eqref{Eq:SquaredErr} to the non-squared ones.

\subsection{ Comparison}
\label{Sec:Implementation}

In recent years, various packages have been developed to facilitate the computation of the electronic structure of materials. See e.g. \cite{GulKMetal14,DesSSetal12,SanFMetal19}, or \url{https://www.nomad-coe.eu/externals/codes} for an overview. In particular, computing excited states via methods based on many-body-perturbation theory has come into focus, as powerful computational resources become more widely available. Here, the Bethe-Salpeter approach constitutes a state-of-the art method for computing optical properties such as the optical absorption spectrum. To this end, Algorithm \ref{Alg:H1sqrt} is typically used to solve the resulting eigenvalue problem after the matrices $A$ and $B$ have been set up \cite{SanMK14}.

The main contribution of the previous section was to provide a unified frame of reference, which can be used to derive the existing approach (Algorithm \ref{Alg:H1sqrt}) as well as two new ones (Algorithm \ref{Alg:H1chol} and Algorithm \ref{Alg:H1svd}). Due to this unified framework, the similarities between the realizations of the different approaches become apparent. In all algorithms we clearly see  four steps.
\begin{enumerate}
 \item Preprocessing: Setup a matrix $M$.
 \item Decomposition: Compute spectral, respectively, singular value decomposition of $M$.
 \item Postprocessing: Transform resulting vectors to (left and right) eigenvectors of matrix $(A+B)(A-B)$.
 \item Final setup: Form eigenvectors of original BSE matrix. 
\end{enumerate}

A detailed compilation is given in Table \ref{Tab:Processing}.

\begin{table}[H]
\centering
{\renewcommand{\arraystretch}{2}
 \begin{tabular}{p{0.155\textwidth}|p{0.235\textwidth}|p{0.235\textwidth}|p{0.235\textwidth}}
  & SQRT (Alg. \ref{Alg:H1sqrt}) & CHOL (Alg. \ref{Alg:H1chol}) &CHOL+SVD (Alg. \ref{Alg:H1svd}) \\\hline
  \small 1. Pre\-processing  & $S=({A-B})^{\frac{1}{2}}$,\hfill \textcolor{blue}{($11n^3$)}\newline\newline  $M=S(A+B)S$\hfill \textcolor{blue}{($4n^3$)} & $LL^{\herm} = A-B$, \hfill\textcolor{blue}{($\frac{1}{3}n^3$)}\newline\newline $M=L(A+B)L^{\herm}$\hfill \textcolor{blue}{($2n^3$)}&$L_1L_1^{\herm} = A-B$,\hfill\textcolor{blue}{($\frac{1}{3}n^3$)}\newline  $L_2L_2^{\herm} = A+B$, \hfill\textcolor{blue}{($\frac{1}{3}n^3$)}\newline $M=L_1^{\herm}L_2$ \hfill\textcolor{blue}{($n^3$)}\\\hline
  \small2. Decomposition &$M=V_M\Lambda^2 V_M^{\herm}$\hfill \textcolor{blue}{($9n^3$)} & $M=V_M\Lambda^2 V_M^{\herm}$ \hfill\textcolor{blue}{($9n^3$)} & $M=U_{SVD}\Lambda V_{SVD}^{\herm}$ \hfill\textcolor{blue}{($21n^3$)} \\\hline
  \small3. Post\-processing & $V_1 := S^{-1} V_M{\Lambda}^{\frac{1}{2}}$, \hfill\textcolor{blue}{($\frac{8}{3}n^3$)}\newline
 $V_2 = S V_M{\Lambda}^{-\frac{1}{2}}$ \hfill\textcolor{blue}{($2n^3$)}&
 $V_1= L^{-H}V_M{\Lambda}^{\frac{1}{2}}$, \hfill\textcolor{blue}{($n^3$)}\newline
$V_2= LV_M{\Lambda}^{-\frac{1}{2}}$ \hfill\textcolor{blue}{($n^3$)}
 &$V_1 = L_1U_{SVD} {\Lambda}^{-\frac{1}{2}}$, \hfill\textcolor{blue}{($n^3$)}\newline
$V_2 = L_2 V_{SVD}{\Lambda}^{-\frac{1}{2}}$ \hfill\textcolor{blue}{($n^3$)}\\\hline
  \small4. Final setup&\multicolumn{3}{c}{\renewcommand{\arraystretch}{1}
  $V=\begin{bmatrix}
    \frac{1}{2}(V_1+V_2)\\
    \frac{1}{2}(V_2-V_1)
   \end{bmatrix}$.}
 \end{tabular} 
 }
\caption{Algorithmic steps of the different methods. The number in brackets estimates the number of flops, where lower-order terms are neglected.\label{Tab:Processing}}
\end{table} 
\renewcommand{\arraystretch}{1}

Seeing the algorithms side by side enables a direct comparison. The amount of flops is based on estimates for sequential, non blocked implementations \cite{GolV13}, and lower order terms i.e. $\mathcal{O}(n^2)$ and $\mathcal{O}(n)$, are neglected. The preprocessing step is most expensive in the square root approach. Computing the square root of a Hermitian matrix involves the solution of a Hermitian eigenvalue problem. Additionally, the matrices $S$ and $M$ need to be set up, using 3 matrix-matrix products.  This makes the preprocessing step even more expensive as the following ``main'' eigenvalue computation. The CHOL and the CHOL+SVD approach, on the other hand, only rely on one or two Cholesky factorizatons and matrix multiplications, which are comparatively cheap to realize. The computational effort in the decomposition step is the highest in the CHOL+SVD step. The post-processing step again is most expensive in the SQRT approach, because the matrix $S$ is a general square matrix, while the $L$ matrices in CHOL and CHOL+SVD are triangular. In total, SQRT takes an estimated amount of $28\frac{2}{3}n^3$ flops, $CHOL$ takes $13\frac{1}{3}n^3$ flops and CHOL+SVD takes $24\frac{2}{3}n^3$ flops. The classical QR algorithm applied to the full, non-Hermitian matrix takes about $25(2n)^3=200n^3$ flops (not including the computation of eigenvectors from the Schur vectors). Solving the Hermitian-definite eigenvalue problem \eqref{Eq:BSEPencil} can exploit symmetry, but still acts on the large problem and can be expected to perform $14(2n)^3=112n^3$ flops. 

According to this metric, we expect both new approaches to perform faster than the square root approach. The actual performance of algorithms on modern architectures is not simply determined by the number of operations performed, but by their parallelizability and communication costs. All presented approaches have a high computational intesity of $\mathcal{O}(n^3)$, such that the memory bandwith is not likely to be a bottleneck. All methods rely on the same standard building blocks from numerical linear algebra, for which optimized versions (e.g. blocked variants for cache-efficiency) are available. This setting makes a fair comparison possible where the arithmetic complexity has a high explanatory power. 

To summarize, we expect CHOL to be about twice as fast as SQRT, while keeping the same accuracy. CHOL+SVD performs more computations than CHOL, and will take more time, but could improve the accuracy of the computations. It might be faster than SQRT, depending on how efficient the diagonalizations in SQRT and the SVD in CHOL+SVD are implemented. 

The comparison in Table \ref{Tab:Processing} is helpful when implementing the new approaches in codes that already use the square root approach. For the Cholesky approach we need to substitute the computation of the matrix square root with the computation of a Cholesky factorization (LAPACK routine \texttt{zpotrf}), compute the matrix $M$ using triangular matrix multiplications (\texttt{ztrmm}), and use a triangular solve (\texttt{ztrsm}) and a triangular matrix multiplication (\texttt{ztrmm}) in the post-processing step. For the CHOL+SVD approach, an additional Cholesky factorization is necessary and the Hermitian eigenvalue decomposition is substituted by a singular value decomposition (\texttt{zgesvd}). The post-processing involves two triangular matrix products instead of a matrix inversion and two general matrix products.

\section{Numerical Experiments}

We implemented and compared serial versions of Algorithms  \ref{Alg:H1sqrt}, \ref{Alg:H1chol} and \ref{Alg:H1svd} in MATLAB. They compute positive eigenvalues and associated eigenvectors of a BSE matrix $H\in\mathbb{C}^{2n\times 2n}$ of form I   \eqref{Eq:Hform1}, which fulfills the definiteness property $\Sigma H >0$ discussed in Section \ref{Sec:Theory}. The eigenvalues are given as a diagonal matrix $D\in\mathbb{R}^{n\times n}$. The eigenvectors  $V\in\mathbb{C}^{2n\times n}$ are scaled s.t. $\Sigma$-orthogonality holds, i.e. $V^{\herm}\Sigma V = I_n$. The $\Sigma$-orthogonality is an important property in the application. It is exploited in order to construct the polarizability operator ultimately used for the computation of the absorption spectrum. 

We also include the MATLAB eigensolver \texttt{eig} for comparison. \texttt{eig} can either work on the BSE matrix $H$ or solve the generalized eigenvalue problem for the matrix pencil $(\Sigma H, \Sigma)$. In this formulation, both matrices are Hermitian and one is positive definite, which allows for a faster computation. 

The experiments were performed on a laptop with an Intel(R) Core(TM) i7-8550U processor using MATLAB R2018a.

The first experiments aim to assess the accuracy of the computed eigenvalues. The matrices $A$ and $B$ are of size $n=200$ and are created in the following way for a given value $\kappa\in\mathbb{R}$. Let $d=\begin{bmatrix}1,\dots,\frac{1}{3}\kappa\end{bmatrix}\in\mathbb{R}^{n}$ be a vector with elements equally spaced between $1$ and $\frac{1}{3}\kappa$. The BSE matrix is constructed as
\begin{align*}
 H = \begin{bmatrix}
      A & B \\
      -B & -A
     \end{bmatrix} := 
     \begin{bmatrix}
      Q&0\\
      0&Q 
     \end{bmatrix}^{\herm}
     \begin{bmatrix}
      \diag{d}&\frac{1}{2}\diag{d}\\
      -\frac{1}{2}\diag{d}&-\diag{d}
     \end{bmatrix}\begin{bmatrix}
      Q&0\\
      0&Q 
     \end{bmatrix},  
\end{align*}
where $Q\in\mathbb{C}^{n\times n}$ is a randomly generated, unitary matrix. It can be shown, that $\cond{H}=\kappa$ and the eigenvalues are given as $\frac{\sqrt{3}}{2}d$. 

\begin{table}
\centering
\begin{tabular}{lccccr}
\toprule
Method &\multicolumn{4} {c}{Relative Error}  & Runtime\\
\hline
&$\kappa=10$ & $\kappa=10^3$&$\kappa=10^6$ & $\kappa=10^9$&\\
\cmidrule{2-5}
\texttt{eig}            & 1.28e-14&5.08e-14&3.82e-11&1.26e-08&62.7 ms\\
generalized \texttt{eig}& 7.89e-15&6.67e-15&1.89e-11&1.97e-09&10.7 ms\\
\texttt{haeig}          & 4.73e-15&7.82e-15&4.32e-11&2.23e-08&50.9 ms\\
SQRT                    & 5.45e-15&3.11e-12&4.64e-06&1.39e+00&5.87 ms\\
CHOL                    & 4.23e-15&2.17e-12&1.32e-06&1.19e-05&3.09 ms\\
CHOL + SVD              & 1.23e-15&2.20e-14&2.53e-11&2.38e-09&4.28 ms\\\bottomrule
\end{tabular}
\caption{Comparison of different methods for eigenvalue computation for Bethe-Salpeter matrix of form I of size $400\times 400$.}\label{Tab:SmallEV}
\end{table}

Table \ref{Tab:SmallEV} shows the relative error in the smallest eigenvalue $\lambda=\frac{\sqrt{3}}{2}$, using the methods discussed in Section \ref{Sec:Algs}. We also included the routine \texttt{haeig} from the SLICOT package \cite{BenMX98b, BenKSetal10}. Because \texttt{haeig} can only compute eigenvalues, not eigenvectors, we also only compute eigenvalues in the other methods in order to make the runtimes comparable.

The MATLAB \texttt{eig} function has the largest runtime. \texttt{haeig} is slightly faster, because it exploits the available Hamiltonian structure. However, the routine is not optimized for cache-reuse, which is why this effect can not be observed more clearly and vanishes for larger matrices. The generalized eigenvalue problem can be solved much faster, because it can be transformed to a Hermitian eigenvalue problem of size $2n\times 2n$. The other methods ultimately act on Hermitian matrices of size $n\times n$, which explains the much lower runtimes. 

The observed eigenvalue errors also comply with the error analysis given in Section \ref{Sec:SVD}. The state-of-the-art square root approach performs even worse than expected, yielding a completely wrong eigenvalue for matrices with a condition number $\kappa=10^9$. In the application context, the small eigenvalues are of special interest. They correspond to bound exciton states, representing a strong electron-hole interaction. They are the reason why the Bethe-Salpeter approach is used instead of simpler schemes based on time-dependent density functional theory \cite{SagA09}. The smallest eigenvalues suffer the most from this numerical inaccuracy.

The second experiment aims to asses the runtime of the sequential implementations, including the eigenvector computation in the measurement. The matrices $A$ and $B$ are setup as random matrices, where the diagonal of $A$ has been scaled up in order to guarantee the definiteness property $\Sigma H >0$. The measured runtimes are found in Figure \ref{Fig:Runtimes} and serve as a rough indicator of computational effort. 

 \newlength\figureheight
\newlength\figurewidth

\begin{figure}[H]
\centering
\setlength\figureheight{0.4\textwidth}
\setlength\figurewidth{0.7\textwidth}
%
%
\definecolor{mycolor1}{rgb}{0.00000,0.44700,0.74100}%
\definecolor{mycolor2}{rgb}{0.85000,0.32500,0.09800}%
\definecolor{mycolor3}{rgb}{0.92900,0.69400,0.12500}%
\definecolor{mycolor4}{rgb}{0.49400,0.18400,0.55600}%
\definecolor{mycolor5}{rgb}{0.46600,0.67400,0.18800}%
\begin{tikzpicture}

\begin{axis}[%
width=0.951\figurewidth,
height=\figureheight,
at={(0\figurewidth,0\figureheight)},
scale only axis,
xmin=0,
xmax=2000,
xlabel={Matrix size $n$},
ylabel={Runtime in seconds},
ymin=0,
ymax=120,
axis background/.style={fill=white},
title style={font=\bfseries},
axis x line*=bottom,
axis y line*=left,
legend style={at={(0.1,0.6)},font=\tiny, anchor=south west,legend cell align=left, align=left, draw=white!15!black}
]
\addplot [color=mycolor1, mark=o, mark options={solid, mycolor1}]
  table[row sep=crcr]{%
10	0.010026\\
50	0.142951\\
100	0.965847\\
200	2.816384\\
500	14.073009\\
1000	66.380098\\
};
\addlegendentry{MATLAB eig}

\addplot [color=mycolor2, mark=o, mark options={solid, mycolor2}]
  table[row sep=crcr]{%
10	0.016438\\
50	0.06388\\
100	0.210707\\
200	0.396942\\
500	3.28118\\
1000	16.229203\\
1500	47.006843\\
2000	101.396473\\
};
\addlegendentry{MATLAB generalized eig}

\addplot [color=mycolor3, mark=o, mark options={solid, mycolor3}]
  table[row sep=crcr]{%
10	0.023727\\
50	0.023042\\
100	0.13059\\
200	0.43694\\
500	1.960937\\
1000	7.14946\\
1500	17.950824\\
2000	35.856785\\
};
\addlegendentry{SQRT approach}

\addplot [color=mycolor4, mark=o, mark options={solid, mycolor4}]
  table[row sep=crcr]{%
10	0.021901\\
50	0.008087\\
100	0.068491\\
200	0.218528\\
500	1.02055\\
1000	3.709501\\
1500	10.116111\\
2000	21.062858\\
};
\addlegendentry{CHOL approach}

\addplot [color=mycolor5, mark=o, mark options={solid, mycolor5}]
  table[row sep=crcr]{%
10	0.018945\\
50	0.006926\\
100	0.024317\\
200	0.370956\\
500	1.463673\\
1000	6.02471\\
1500	15.801772\\
2000	32.537312\\
};
\addlegendentry{CHOL+SVD approach}

\end{axis}
\end{tikzpicture}%
 \caption{ Runtimes for different methods, $A,B\in\mathbb{C}^{n\times n}$ with varying matrix sizes. \label{Fig:Runtimes}}
\end{figure}
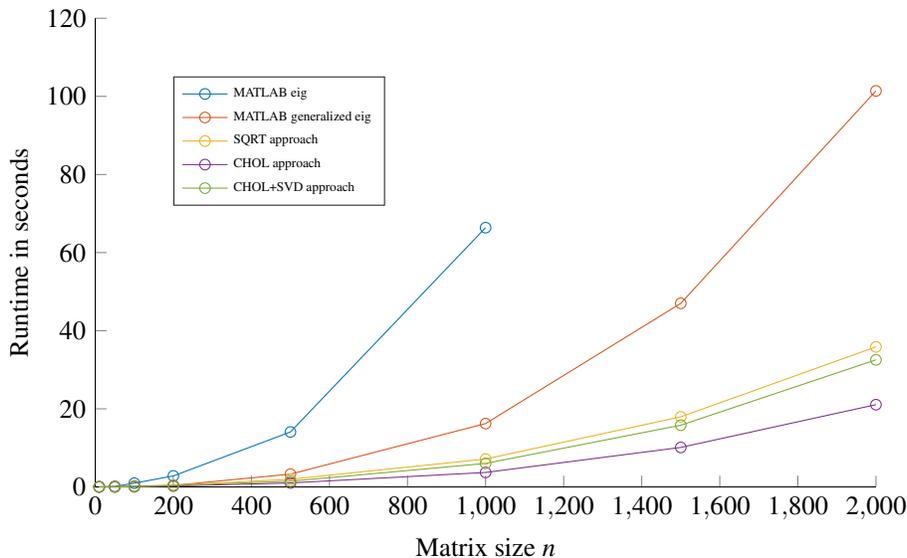

As expected, the Cholesky approach yields the fastest runtime of all approaches. The SVD approach also performs better than the square root approach. However, this picture could easily look different in another computational setup. 
%
 An approach based on the \texttt{eig} command becomes prohibitively slow, when larger matrices are considered. Matrices in real applications become extremely large, up to dimensions of order $100\,000$,  in order to get reasonable results.
 The effect would be even more drastic in a parallel setting, as the solution of a nonsymmetric dense eigenvalue problem is  notoriously difficult to parallelize.


Figure \ref{Fig:Sigma} shows the achieved $\Sigma$-orthogonality of the eigenvector matrices for matrices with certain condition numbers. To this end, we manipulate the diagonal of the randomly generated matrix $A$ such that badly conditioned BSE matrices $H$ are generated. For the square root and the Cholesky approach, the $\Sigma$-orthogonality breaks down completely for badly conditioned matrices. This can have dramatic consequences and lead to completely wrong results, when further computations rely on this property. 

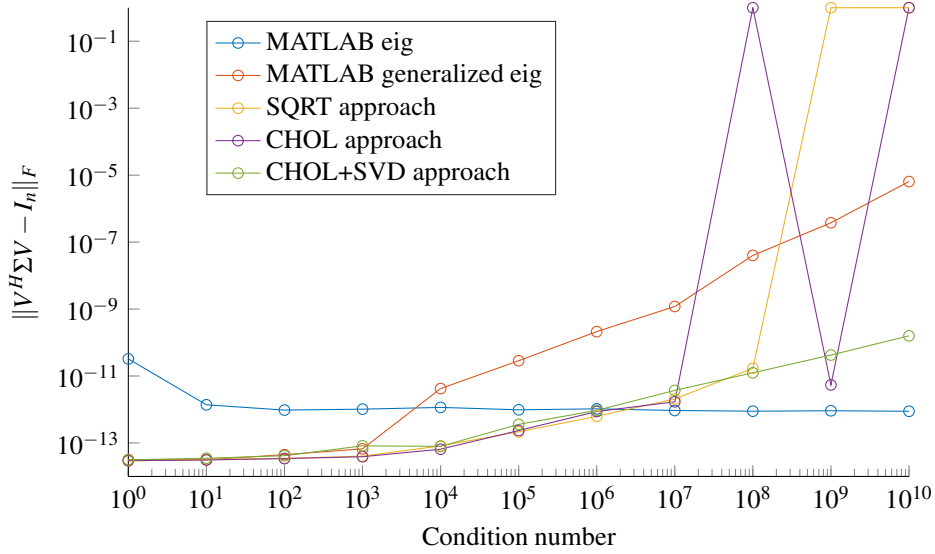
\begin{figure}[h] 
\centering
\setlength\figureheight{0.4\textwidth}
\setlength\figurewidth{0.7\textwidth}
%
%
\definecolor{mycolor1}{rgb}{0.00000,0.44700,0.74100}%
\definecolor{mycolor2}{rgb}{0.85000,0.32500,0.09800}%
\definecolor{mycolor3}{rgb}{0.92900,0.69400,0.12500}%
\definecolor{mycolor4}{rgb}{0.49400,0.18400,0.55600}%
\definecolor{mycolor5}{rgb}{0.46600,0.67400,0.18800}%
\begin{tikzpicture}

\begin{axis}[%
width=0.951\figurewidth,
height=\figureheight,
at={(0\figurewidth,0\figureheight)},
scale only axis,
xmode=log,
xmin=1,
xmax=10000000000,
xminorticks=true,
xlabel={Condition number},
ylabel={$\|V^H\Sigma V - I_n\|_F$},
ymode=log,
ymin=1e-14,
ymax=1.00000000000001,
yminorticks=true,
axis background/.style={fill=white},
title style={font=\bfseries},
axis x line*=bottom,
axis y line*=left,
legend style={at={(0.1,0.6)}, anchor=south west, legend cell align=left, align=left, draw=white!15!black}
]
\addplot [color=mycolor1, mark=o, mark options={solid, mycolor1}]
  table[row sep=crcr]{%
1	3.24839701354267e-11\\
10	1.38130134657597e-12\\
100	9.58630551148589e-13\\
1000	1.02120443005206e-12\\
10000	1.15325814329188e-12\\
100000	9.77364092781582e-13\\
1000000	1.04462529068071e-12\\
10000000	9.32247535122271e-13\\
100000000	8.88459086274439e-13\\
1000000000	9.1527218975111e-13\\
10000000000	8.83484947520461e-13\\
};
\addlegendentry{MATLAB eig}

\addplot [color=mycolor2, mark=o, mark options={solid, mycolor2}]
  table[row sep=crcr]{%
1	3.17750055355442e-14\\
10	3.14401044806276e-14\\
100	4.47019975504983e-14\\
1000	6.66545181436257e-14\\
10000	4.20433349469036e-12\\
100000	2.86166067940511e-11\\
1000000	2.11603968406753e-10\\
10000000	1.19139936242437e-09\\
100000000	3.9924789995784e-08\\
1000000000	3.77402979901346e-07\\
10000000000	6.42897709069732e-06\\
};
\addlegendentry{MATLAB generalized eig}

\addplot [color=mycolor3, mark=o, mark options={solid, mycolor3}]
  table[row sep=crcr]{%
1	2.91705643229762e-14\\
10	3.23462665928615e-14\\
100	3.43215053818352e-14\\
1000	4.03723820507963e-14\\
10000	8.15642772330854e-14\\
100000	2.14657451991964e-13\\
1000000	6.24290798835974e-13\\
10000000	2.06745852439192e-12\\
100000000	1.69904112503673e-11\\
1000000000	1\\
10000000000	1.00000000000001\\
};
\addlegendentry{SQRT approach}

\addplot [color=mycolor4, mark=o, mark options={solid, mycolor4}]
  table[row sep=crcr]{%
1	2.95568585263705e-14\\
10	3.09878576470553e-14\\
100	3.41688223883668e-14\\
1000	3.87485500910984e-14\\
10000	6.51197433464706e-14\\
100000	2.35005143104319e-13\\
1000000	8.67281077912016e-13\\
10000000	1.69550413582318e-12\\
100000000	1\\
1000000000	5.42657640545693e-12\\
10000000000	1.00000000000001\\
};
\addlegendentry{CHOL approach}

\addplot [color=mycolor5, mark=o, mark options={solid, mycolor5}]
  table[row sep=crcr]{%
1	3.12208064950897e-14\\
10	3.50023090122589e-14\\
100	4.15347277967316e-14\\
1000	8.18647554639377e-14\\
10000	7.91018019353045e-14\\
100000	3.59377181722864e-13\\
1000000	9.40225542285529e-13\\
10000000	3.71202664913919e-12\\
100000000	1.24721212271724e-11\\
1000000000	4.22452466090451e-11\\
10000000000	1.58402600215786e-10\\
};
\addlegendentry{CHOL+SVD approach}

\end{axis}
\end{tikzpicture}%
 \caption{ Deviation from $\Sigma$-orthogonality for different methods, $A,B\in\mathbb{C}^{200\times 200}$ with a certain condition number. \label{Fig:Sigma}}
\end{figure}


%
%
%

To show the applicability to real life examples, we extracted a Bethe-Salpeter matrix corresponding to the excitation of Lithium-Fluoride from the \texttt{exciting} software package \cite{GulKMetal14}. Computational details on how the matrix is generated can be found in the documentation\footnote{\url{http://exciting-code.org/carbon-excited-states-from-bse}}. Here, it is pointed out that a Tamm-Dancoff approximation, i.e. setting the off-diagonal block $B$ to zero, already yields satisfactory results. The resulting $2560\times2560$ BSE matrix has a condition number (computed using \texttt{cond} in MATLAB)  of $5.33$. We do not expect the algorithms to suffer from the numerical difficulties observed in the first example. 

\begin{table}
\resizebox{\textwidth}{!}{
\begin{tabular}{lllllll}
\toprule
&$\lambda_1$&$\lambda_2$&$\lambda_3$& Runtime\\
\texttt{eig}            &4.6423352497493209e-01&4.6524229149750918e-01&4.6872644706731720e-01&32.49 s\\
generalized \texttt{eig}&4.6423352497493126e-01&4.6524229149750407e-01&4.6872644706732447e-01&10.62 s\\
\texttt{haeig}          &4.6423352497493725e-01&4.6524229149750940e-01&4.6872644706732514e-01&71.43 s\\
SQRT                    &4.6423352497493120e-01&4.6524229149750490e-01&4.6872644706732541e-01&3.44  s\\
CHOL                    &4.6423352497493031e-01&4.6524229149750573e-01&4.6872644706732414e-01&2.06  s\\
CHOL + SVD              &4.6423352497493092e-01&4.6524229149750473e-01&4.6872644706732453e-01&3.41  s\\
TDA                     &4.6427305979874345e-01&4.6528180480128906e-01&4.6877150201685513e-01&0.88  s\\   \bottomrule
\end{tabular}
}
\caption{Computed eigenvalues for Lithium Fluoride example.}\label{Tab:LiF}
\end{table}

The three smallest eigenvalues computed by different methods are found in Table \ref{Tab:LiF}. Indeed, all approaches coincide in the first 14 significant digits. The Tamm-Dancoff approximation (TDA) applies MATLAB \texttt{eig} on the diagonal Block $A$ and provides eigenvalues, that are correct up to 4 significant digits which is sufficient for practical applications. The measured runtimes reflect the results of the other experiments. Now the lack of low-level optimization in the \texttt{haeig} routine becomes apparent and leads to the lowest performance of all approaches.

\section{Conclusions}

We presented two new approaches for solving the Bethe-Salpeter eigenvalue problem as it appears in the computation of optical properties of crystalline systems. The presented methods are superior to the one currently used, which is based on the computation of a matrix square root. Computing the matrix square root constitutes a high computational effort for nondiagonal matrices. Our first proposed method substitutes the matrix square root with a Cholesky factorization which can be computed much easier. The total runtime is reduced by about 40\% in preliminary experiments, while the same accuracy is achieved. In order to achieve a higher accuracy we proposed a second method, which also relies on Cholesky factorizations and uses a singular value decomposition instead of an eigenvalue decomposition. 


We also gave new theoretical results on structured matrices, which served as a foundation of the proposed algorithms.


\addcontentsline{toc}{section}{References}

\end{document}